\documentclass[twoside,11pt,reqno]{amsart}

\usepackage[bookmarks,
bookmarksnumbered,%
 colorlinks=true,%
 linkcolor=red,%
 citecolor=blue,%
 filecolor=blue,%
 urlcolor=blue,%
]
{hyperref}
\usepackage{amssymb, amscd,fge}
\usepackage{xypic}
\usepackage{epigraph}

\DeclareMathOperator{\hilb}{Hilb}

\DeclareMathOperator{\Pic}{Pic}
\DeclareMathOperator{\rank}{rank}

\DeclareMathOperator{\im}{Im}

\def\dim{\mbox{dim}}
\def\ra{\rightarrow}
\def\cal{\mathcal} 

\def\ol{\overline}

\def\CC{\mathbb{C}}
\def\PP{\mathbb{P}}

\def\ZZ{\mathbb{Z}}
\def\NN{\mathbb{N}}  
\def\RR{\mathbb{R}}

\def\HH{\mathbb{H}}

\def\FF{\cal F}

\def\OO{\cal O}

\def\GG{\cal G}

\def\s-{\setminus}

\newtheorem{main}{Theorem}

\newtheorem{thm}{Theorem}[section]
\newtheorem{prop}[thm]{Proposition}

\newtheorem{cor}[thm]{Corollary}
\newtheorem{lem}[thm]{Lemma}

\newtheorem{rmk}[thm]{Remark}

\numberwithin{equation}{section}

\newenvironment{notations}{\mbox{ }\\{\bf  Terminological convention:}\mbox{ }}

\def\ps{\vspace{4pt}}

\begin{document}

\title{Counting real rational curves on $K3$ surfaces}

\author[Kharlamov]{Viatcheslav Kharlamov}

\address{
        IRMA UMR 7501, Universit\'e de Strasbourg, 7 rue Ren\'e-Descartes, 67084 Strasbourg Cedex,  FRANCE}

\email{kharlam@math.unistra.fr}

\author[R\u asdeaconu]{Rare\c s R\u asdeaconu}

\address{        
        Department of Mathematics, 1326 Stevenson Center, Vanderbilt University, Nashville, TN, 37240, USA}

\email{rares.rasdeaconu@vanderbilt.edu}

\keywords{$K3$ surfaces, compactified Picard variety, rational curves, Yau-Zaslow formula, Welschinger invariants}

\subjclass[2000]{Primary: 14N99; Secondary: 14P99, 14J28}

\date{}

\begin{abstract}
We provide a real analog of the Yau-Zaslow formula counting rational curves on $K3$ surfaces.
\end{abstract}

\maketitle

\thispagestyle{empty}

\epigraph{\scriptsize "But man is a fickle and disreputable creature and perhaps, like a chess-player, 
is interested in the process of attaining his goal rather than the goal itself."}
{\it Fyodor Dostoyevsky, Notes from the Underground.}

\section*{Introduction}
\label{sec.intro}

We suggest, as a real analog of the Yau-Zaslow formula \cite{yau-zaslow}, 
a certain signed count of real rational curves in a primitive 
divisor class on a generic real $K3$ surface. Our count is invariant 
under deformations of the surface and yields a non-trivial lower 
bound for the number of real rational curves in question. The goal 
is achieved by following Beauville's approach \cite{beauville} 
towards the original Yau-Zaslow  prediction  for the number of complex 
rational curves on generic $K3$ surfaces. In the real setting, 
some of the steps lead to results that may be of independent interest.

We prove, first, that for a real curve of positive geometric genus, the Euler 
characteristic of each of its compactified Picard varieties vanishes, while
for a real rational nodal curve it is not zero and equal to $(-1)^gw$, where 
$g$ is the genus and $w$ the Welschinger number \cite{welschinger}
of the curve. The latter 
is equal to $(-1)^s$, where $s$ is the number of zero-dimensional components
in the real locus of the curve. 
Counting the real curves with Welschinger signs we arrive to
the following statement relating such a signed count with the Euler 
characteristic of  the real locus of a punctual Hilbert scheme.

\begin{main}
\label{main}
Let $X$ be a generic real $K3$ surface admitting a complete real 
$g$-dimensional linear system of curves of genus $g.$
If $g\ge 2,$ assume, in addition, that $X$ is of Picard rank $1$
and the curves in the linear system belong to a primitive divisor class.
Let $n_+$ (respectively, $n_-$) denote the number of real rational curves in
 this linear system that have
Welschinger sign plus (respectively, minus). Then, 
$$
n_+-n_-=(-1)^g e(X_\RR^{[g]}),\quad  n_++n_-\ge  \vert e(X_\RR^{[g]})\vert,
$$
where $X^{[g]}$ is the Hilbert scheme that parametrizes finite subschemes of length 
$g$ of $X$ as a complex surface,  $X_\RR^{[g]}$ is the real locus
of $X^{[g]},$ and $e$ states for the Euler characteristic.
\end{main}

The next result is a real analog of G\"ottsche's formula \cite{gottsche}. 
It expresses  the generating series for the Euler characteristic of real Hilbert 
schemes of points on complex surfaces equipped with an 
anti-holomorphic involution (in particular, defined over the reals) 
in rather explicit terms.

\begin{main}
\label{realG} For any complex surface $X$ equipped with an anti-holomorphic 
involution, we have:
\begin{equation}
\label{real-gottsche}
\sum_{g\geq 0} e(X_\RR^{[g]} ) q^g=
	\prod_{r\geq 1} \frac{1}{\left(1+(-q)^r\right)^{e_\RR}} 
	\prod_{s\geq 1} \frac{1}{(1-q^{2s})^{\frac{e_\CC-e_\RR}2}},
\end{equation}
where $e_\RR=e(X_\RR)$ and $e_\CC=e(X).$
\end{main}

An immediate consequence of Theorems \ref{main} and \ref{realG} is the 
following real version of the Yau-Zaslow formula:
\begin{cor}
\label{ryz}
Let $X$ be a generic real $K3$ surface admitting a complete real $g$-dimensional 
linear system of curves of genus $g.$ If $g\ge 2,$ assume, in addition, that $X$ is of 
Picard rank $1$ and the curves in the linear system belong to a primitive divisor class.
Let $w_g=n_+-n_-$ denote the number of real rational curves in this linear system 
counted with Welschinger sign. Then, 
\begin{enumerate}
\item The number $w_g$ depends only on genus $g$ and the topology of $X_\RR$ 
and, in particular, is invariant under equivariant (with respect to complex conjugation) 
deformations of $X.$
\item The generating function for $w_g$ is as follows:
$$
\sum_{g\geq 0} w_g q^g=
	\prod_{r\geq 1}\frac {1}{(1+q^r)^{e_\RR}}
	\prod_{s\geq 1}\frac {1}{(1-q^{2s})^{\frac{e_\CC-e_\RR}{2}}}.
$$
\end{enumerate}
\end{cor}

\notations In our calculation of Euler characteristic of real varieties we use, 
in accordance with the conventions specific for integration with respect to 
Euler characteristic \cite{viro}, the cohomology with compact supports 
(and rational coefficients). This convention allows us to apply
the additivity property, $e(X)=e(X\setminus Y)+e(Y)$, to any closed subset $Y$ of $X.$

\section{The relative compactified Picard varieties}
\label{sec: rcp}

\subsection{Preliminaries}
We start our study of the enumerative geometry of real $K3$ surfaces that admit 
a real $g$-dimensional system of curves of genus $g$ by an investigation of the 
relative compactified Picard variety of degree $d=g$ associated with such a linear
system of curves.

Recall that, over the complex field (as well as over any algebraically closed field), 
for any flat projective family ${\cal C}/S$ of reduced irreducible curves over a  
reduced irreducible base, and any degree $d,$ there exists a projective family 
$\ol{\Pic}^d{\cal C}/S,$ whose fibers are the compactified  Picard varieties  
$\ol{\Pic}^d(C_s)$ parametrizing isomorphism classes of rank 1, torsion-free, 
degree $d,$ coherent sheaves on $C_s, ~s\in S.$  Moreover, if $C_s$ 
has only planar singularities, then $\ol{\Pic}^d(C_s)$ is reduced and irreducible 
(see, for example, \cite{souza} and \cite{a-k}).

Let $X$ be a $K3$ surface admitting a $g$-dimensional linear system $\cal C$ 
of curves of genus $g$ and let us assume, as in Theorem \ref{main},  that the  
N\'eron-Severi group of $X$ is $\ZZ$ and the curves in this linear system provide 
the generator of the N\'eron-Severi group. Then all the curves in the linear system 
are reduced and irreducible, and they form a flat projective family. In this case, 
the relative compactified Picard varieties $\ol{\Pic}^d \cal C$ can be identified 
with connected components of the moduli space of simple sheaves on $X,$
as proved by Mukai \cite{mukai}. Furthermore, the varieties $\ol{\Pic}^d \cal C$ 
are nonsingular, see \cite{mukai}.

When $X$ is equipped with a real structure, i.e., an anti-holomorphic involution $c :X\to X,$ 
we can associate with it a natural real structure on $\ol{\Pic}^d \cal C.$ 
That is an anti-holomorphic  involution 
${c}_{\ol{\Pic}^d\cal C}:\ol{\Pic}^d \cal C\ra \ol{\Pic}^d \cal C$ 
that transforms  sheaves ${\FF}$ into sheaves  $c^*{\ol{\FF}}=\ol{c^*\FF}.$

The following Lemma is well known in the case of line bundles, but we did not find 
in the literature a corresponding statement for sheaves in the required form.

\begin{lem}
\label{Mukai-stronglyreal}  
If a sheaf $\FF$ representing a point of 
$\ol{\Pic}_\RR^d \cal C=\operatorname{Fix}(c_{\ol{\Pic}^d\cal C})$
is not defined over reals, then:
\begin{enumerate} 
\item $X_\RR$ is empty.
\item $\FF$ carries an anti-automorphism $c_\FF :\FF\to \FF$ with $c^2_\FF=-1.$
\end{enumerate}
In particular, under assumption that $X_\RR$ is non empty, a sheaf 
$\cal F\in \ol{\Pic}^d\cal C$ belongs to 
$\ol{\Pic}_\RR^d\cal C=\operatorname{Fix}(c_{\ol{\Pic}^d\cal C})$ 
if and only if it is defined  over $\RR$ 
(equivalently, carries an anti-automorphism of period 2). 
\end{lem}

\proof Clearly, the sheaves  $\FF \in \ol{\Pic}^d\cal C$ that are defined over $\RR$
belong to  $\ol{\Pic}_\RR^d\cal C.$

Assume that $\FF$ belongs to $\ol{\Pic}_\RR^d \cal C.$ 
Let us interpret an isomorphism between
${\FF}$ and  $\ol{c^*\FF }$ as an anti-automorphism $A$ of $\FF$ that lifts $c:X\to X$. 
Over each germ $(x,U)$ of $X$ pick a generator
$f_U\in \FF(U)$. With respect to these generators, the above anti-automorphism 
is given by a transformation with the property 
$\phi\cdot f_U\in \FF (U)\mapsto \ol{\phi\circ c}\cdot g_U\in\FF(c(U))$, 
where $g_{U}=h_Uf_{c(U)}\in\FF (c(U))$ denotes the image 
$A(f_U)\in\FF(c(U))$ of $f_U\in \FF (U).$

If $X_\RR\neq\emptyset,$ then pick $x\in X_\RR$ and choose $U=c(U).$ 
Applying twice the above transformation to $f_U$, we obtain
$f_U\mapsto \ol{h_U\circ c}\cdot h_U\cdot f_U.$ Since $\FF$ is a simple sheaf, 
$A^2$ (as any endomorphism of $\FF$) is a multiplication by a constant.
This constant should be equal to $\ol{h_U\circ c}\cdot h_U$, hence, it is positive. 
Therefore, dividing $A$ by a square root of this constant
we get an anti-automorphism of $\FF$ of period 2.

If $X_\RR=\emptyset,$  then the constant represented by $A^2$ gets the 
local expressions $\ol{h_U\circ c}\cdot h_{c(U)}$ and $\ol{h_{c(U)}\circ c}\cdot h_{U}$. 
Hence, it is at least real. Therefore, dividing $A$ by a square root of this constant we get an 
anti-automorphism of period 2 or 4. In the first case, $\FF$ is defined over the reals.
\qed

\ps

The canonical regular morphism $\ol{\Pic}^d\cal C\to \vert\cal C\vert$ is equivariant 
with respect to ${c}_{\ol{\Pic}^d\cal C}$ and the real structure $c_{\vert\cal C\vert}$ 
induced on $\vert\cal C\vert$ by $c.$ Hence, we get a regular real algebraic map from 
the real locus $\ol{\Pic}_\RR^d\cal C=\operatorname{Fix}(c_{\ol{\Pic}^d\cal C})$ 
to the real part $\vert\cal C\vert_\RR=\operatorname{Fix}(c_{\vert\cal C\vert})$ of our 
linear system with the real loci of the compactified Picard schemes $\ol{\Pic}^d({C}_s)$ 
as fibers over real points $s\in \vert\cal C\vert_\RR$, where ${C}_s$ denotes the 
fiber over $s$ of $\ol{\Pic}^d\cal C\to \vert\cal C\vert.$

Thus, we can apply Viro-Fubini theorem \cite{viro} to this regular real algebraic map and, 
in such a way, reduce the calculation of the Euler characteristic of $\ol{\Pic}_\RR^d\cal C$ 
(which is one of the objectives in this note)  to integrating of the Euler characteristic of the fibers.
So, we devote the rest of Section \ref{sec: rcp} to evaluating the Euler characteristic of the real 
part for compactified Picard varieties of real curves.

\subsection{Euler characteristic of real compactified Picard varieties of a real curve.}
\label{curve-calculation}

\ps

Let $Y$ be an irreducible, reduced curve of arithmetic genus $g,$ 
defined over $\RR.$ The Galois group $G=Gal(\CC/\RR)$ acts on the
locus of its complex points, which is still be denoted by $Y.$
The locus of real points, which we denote by $Y_\RR,$ is nothing but the fixed 
point set of the action. Conversely, each curve defined over $\RR$ can be 
encoded as a pair $(Y,c),$ where $c$ is an anti-holomorphic involution on 
$Y,$ the locus of complex points of the curve.

The action of $G$ lifts to the Jacobian variety $\Pic^0(Y)$ of the complexification of the 
curve $Y_\RR.$ The fixed point set of this action, $\Pic_\RR^0(Y),$
consists of the $G$-invariant complex divisor classes of degree $0.$ This abelian variety 
contains the subgroup $\left(\Pic^0_\RR(Y)\right)_+$ of those divisor classes 
that can be represented by $G-$invariant divisors of degree $0.$ Notice that the elements of 
$\left(\Pic^0_\RR(Y)\right)_+$ 
correspond to line bundles which are defined over $\RR.$

\begin{prop}\label{nullity}
Let $Y$ be a real, reduced, irreducible curve of positive geometric genus.
Then, each of the real compactified Picard varieties, $\ol{\Pic}^d_{\RR}(Y),$ 
if it is nonempty, has zero Euler characteristic. 
\end{prop}

\begin{proof}
The arguments of Beauville \cite[Proposition 2.2 and Lemma 2.1]{beauville} 
can be adapted for curves defined over $\RR.$  

\ps 

Let $n:Y'\ra Y$ denote the normalization of the curve $Y.$ 
There exists a unique anti-holomorphic involution $c'$ on $Y',$ 
such that $n\circ c'=c\circ n.$ 
Notice that $Y'_\RR$ is non-empty if $\dim Y_\RR=1$ 
(equivalently, $ Y_\RR$ contains smooth points),
but it may be empty when $Y_\RR$ has only 
$0-$dimensional components.

Regardless whether the curve $Y_\RR$ has smooth real points or not \cite[Proposition 21.8.5]{ega4},
we have an exact sequence of group schemes over $\RR:$
$$
0\ra H\ra \left(\Pic_{\RR}^0(Y)\right)_+\xrightarrow{n^*}\left(\Pic_{\RR}^0(Y')\right)_+\ra 0,
$$
where $H$ is a direct product of copies of additive and multiplicative real algebraic groups.
Since $H$ is divisible, the sequence splits, so  $\left(\Pic_{\RR}^0({Y'})\right)_+$ injects into 
$(\Pic_{\RR}^0(Y))_+.$ 
As it follows from \cite[Lemma 2.1]{beauville}, $\left(\Pic_{\RR}^0({Y'})\right)_+$ 
acts freely on $\ol \Pic_{\RR}^d(Y)$ and to observe the vanishing of the Euler characteristic, 
$e(\ol{\Pic}^d_{\RR}(Y))=0,$  there remains to notice that $\left(\Pic_{\RR}^0({Y'})\right)_+$ 
contains at least one real torus, namely, the connected component of the identity.
\end{proof}

Now, let us turn to the case of real rational nodal curves. To understand the topology of the 
compactified Picard variety associated to such curves, we start by taking a closer look
to the complex case.

\ps 

According to \cite{seshadri},  admissible, that is torsion-free, rank $1,$ coherent,    
sheaves $\FF$ along a nodal curve $Y$ defined over $\CC$ satisfy the following conditions:
\begin{itemize}
\item[ a)] If $p\in Y$ is nonsingular, then the stalk $\FF_{p}$ is isomorphic to the stalk 
of the structure sheaf $\OO_{p}.$
\item[ b)] If $p\in Y$ is a node, then either $\FF_{p}$ is either the maximal ideal $m_p,$ 
in which case $\FF$ will be called of Type I at the node $p$, or 
$\FF_{p}\simeq \OO_{p},$ in which case $\FF$ will be called of Type II at the node $p.$
\end{itemize}

Let $n:Y'\ra Y$ be the partial normalization of $Y$ at a node $p,$ and denote by $p_1$ and 
$p_2$ the two points in the pre-image of $p.$ 

\ps

If $\FF$ is a degree $d$ admissible sheaf of Type I at the node $p,$ then the sheaf 
$\FF'=n^*\FF/Torsion$ is an admissible sheaf on $Y'$ and $\FF\simeq n_*\FF'.$ Moreover,
$$
\deg \FF'=\deg\FF-1=d-1.
$$

\begin{lem}[Lemma 3.1 \cite{beauville} and Seshadri \cite{seshadri}]
\label{type1:embed}
Let $n:Y'\ra Y$ be a partial normalization of $Y$ at the nodal point $p.$ 
The morphism 
$$
{n}_*:\ol{\Pic}^{d-1}(Y')\ra \ol{\Pic}^d(Y)
$$ 
is a closed embedding. Moreover,
 $\displaystyle
 \im (n_*)=\{\FF\in \ol{\Pic}^d(Y)~{\text{of Type I at}}~p\}.
 $
 \qed
\end{lem}

\ps

Let $\ol{\Pic}^d_{b}(Y)$ denote this boundary stratum of Type I admissible 
sheaves of degree $d,$ and let $\ol{\Pic}^d_{o}(Y)$ be its complement 
$\ol{\Pic}^d(Y)\setminus  \ol{\Pic}_b^d(Y).$ If  $\FF\in \ol{\Pic}^d_{o}(Y),$ 
then $\FF'=n^*\FF$ is again an admissible sheaf on $Y',$ 
locally free at $p_1$ and $p_2,$ and $\deg \FF'=\deg \FF=d,$ 
i.e., $\FF'\in \ol{\Pic}^d(Y').$ Moreover, there exists an isomorphism 
between that stalks $\FF'_{p_1}$ and  $\FF'_{p_2}.$
Conversely, for any $\FF'\in \ol{\Pic}^d(Y')$ locally free at $p_1$ and $p_2$,
and an isomorphism between $\FF'_{p_1}$ and  $\FF'_{p_2},$ 
there exists a unique $\FF\in \ol{\Pic}^d_{o}(Y)$ such that $\FF'=n^*\FF.$ 
The set of isomorphisms between $\FF'_{p_1}$ and  $\FF'_{p_2}$ is 
identified with $\CC^*$ and the pull-back map
\begin{equation}
\label{open_stratum}
n^*:\ol{\Pic}^d_{o}(Y)\ra \ol{\Pic}^d(Y')
\end{equation}
has a fiber bundle structure with fiber $\CC^*.$

\ps

Assume, now, again that our real reduced irreducible nodal curve $Y$ 
of arithmetic genus $g$ is defined over reals. The nodes of $Y$ lie either 
in $Y_\RR,$ or they come in pairs swapped by the involution $c.$  
The nodal points lying  in $Y_\RR$ 
can be either zero-dimensional components of $Y_\RR,$ called 
{\it solitary points,} or they are nodes of the one-dimensional components 
of $Y_\RR,$ in which case they will be called {\it cross points}. 
In the completion of the local ring, the solitary points are locally 
given by the equation $x^2+y^2=0,$ while the cross points by $x^2-y^2=0.$
Let $r$ denote the number of cross-points, 
$s$  the number of solitary points,
and $t$ the number of complex conjugated pairs of nodal points of $Y.$ 
When the curve $Y$ has geometric genus $0,$ since all the 
singular points of the underlying complex curve are nodes, 
we have $r+s+2t=g.$

The partial normalization $Y'$ at a given node $p\in Y_\RR,$ 
or at a pair of conjugate nodes $\{q, c(q)\}\in Y\setminus Y_\RR$ 
has an induced real structure $c',$ and the partial normalization 
map $n:Y'\ra Y$ is equivariant. Accordingly, the induced maps,   
${n}_*:\ol{\Pic}^{d-1}(Y')\ra \ol{\Pic}^d(Y)$ and 
$n^*:\ol{\Pic}^d_{o}(Y)\ra \ol{\Pic}^d(Y')$ in the case of a real node, 
as well as ${n}_*:\ol{\Pic}^{d-2}(Y')\ra \ol{\Pic}^d(Y)$ and 
$n^*:\ol{\Pic}^d_{o}(Y)\ra \ol{\Pic}^d(Y')$ in the case of a pair 
of conjugate ones, are equivariant as well. In both cases, as well as 
for any normalization defined over reals, the maps $n_*$ and $n^*$ 
send sheaves defined over reals to sheaves defined over reals.

With respect to the real structure on $Y',$ for a real nodal point 
$p,$ the pre-image of $p$ consists of a pair of points $\{p_1, p_2\},$ 
which are either smooth points in $Y'_\RR$ or conjugate smooth points 
with respect to $c',$ depending on whether $p$ is a cross point 
or a solitary point of $Y,$ respectively.  For pairs of conjugate nodes 
$\{q, c(q)\},$  the pre-image consists of two pairs of conjugate smooth 
points $\{(q_1,c'(q_1)), (q_2, c'(q_2))\}.$

From the equivariance of the partial normalization map, we immediately 
get the following real version of Lemma \ref{type1:embed}.
\begin{lem}
\label{rbs} 
Let $n:Y'\ra Y$ be a partial normalization of $Y$ at a real nodal point or 
at a pair of conjugate nodes. 
\begin{itemize} 
\item[ 1)] In the case of a real nodal point, the induced morphisms 
$$
n_*:\ol{\Pic}^{d-1}_{\RR}(Y')\ra \ol{\Pic}^d_{\RR}(Y), 
\quad n_*:(\ol{\Pic}^{d-1}_{\RR}(Y'))_+\ra (\ol{\Pic}^d_{\RR}(Y))_+
$$ 
are closed emdeddings.
\item[ 2)] In the case of a pair of conjugate nodes, the morphisms
$$
n_*:\ol{\Pic}^{d-2}_{\RR}(Y')\ra \ol{\Pic}^d_{\RR}(Y), 
\quad n_*:(\ol{\Pic}^{d-2}_{\RR}(Y'))_+\ra (\ol{\Pic}^d_{\RR}(Y))_+
$$  
are closed emdeddings.
\end{itemize}
In both cases, the images are the sets 
$\ol{\Pic}^d_{\RR, b}(Y)$ and $(\ol{\Pic}^d_{\RR, b}(Y))_+$
of Type I degree $d$ admissible sheaves in 
$\ol{\Pic}^d_{\RR}(Y)$ and $(\ol{\Pic}^d_{\RR}(Y))_+,$ 
respectively.
\qed
\end{lem}

To perform a full recursion in calculating the Euler characteristic 
there remain to control the other parts of the real Picard variety, 
$\ol{\Pic}^d_{\RR,o}(Y)$ and, overall, $(\ol{\Pic}^d_{\RR,o}(Y))_+$.

\begin{lem}
\label{ros}
The mapping $n^*:(\ol{\Pic}^d_{\RR,o}(Y))_+\ra (\ol{\Pic}_\RR^d(Y'))_+$ 
is a fiber bundle with fiber isomorphic to 
\begin{itemize}
\item[ i)] $S^1,$ for partial normalizations at a solitary point,
\item[ ii)] $\RR^*,$ for partial normalizations at a cross point,
\item[ iii)] $\CC^*,$  for partial normalizations at a pair of conjugate nodes.
\end{itemize}
\end{lem}

\begin{proof} 
Recall \cite{seshadri} that $n^*:\ol{\Pic}^d_{o}(Y)\ra \ol{\Pic}^d(Y')$ is a fiber 
bundle whose fibers are the orbits of the canonical action of the kernel $K$ of
$n^*:{\Pic}^0(Y)\ra {\Pic}^0(Y')$ on $\ol{\Pic}^d_{o}(Y).$ Therefore, the fiber 
$\{\FF\in \ol{\Pic}^d_{o}(Y) : n^*\FF=\GG\}$ over  each given $\GG\in \ol{\Pic}^d(Y')$
is canonically identified with the space of pairs $(\GG,\lambda),$ where
$\lambda\in Isom(\GG_{p_1},\GG_{p_2})$ if we normalize one node, and 
$\lambda\in Isom(\GG_{p'_1},$ $\GG_{p'_2}) \times Isom(\GG_{p''_1},\GG_{p''_2}) $ 
if we normalize two nodes. Since $\GG$ is locally free and of rank one at 
the points above the normalized nodes, the set $Isom(\GG_{p_1},\GG_{p_2})$ 
can be canonically identified with $\CC^*,$ in the first case, and with  
$\CC^*\times \CC^*,$ in the second case.

If $n: Y\to Y'$ is  real, the real structure acts naturally on $K,$ so that the 
action of the real part $K_\RR$ of $K$ on $\ol{\Pic}^d_{o}(Y)$ preserves 
the fibration $n^*:(\ol{\Pic}^d_{\RR,o}(Y))_+\ra (\ol{\Pic}_\RR^d(Y'))_+$. 
Since $K$ and, hence, $K_\RR$ act freely, the fibers of the latter fibration 
are canonically identified with $K_\RR$ which by definition is formed by 
$\lambda\in K$compatible with the action of the real structure on the 
stalks of $\cal G$.

In the case of a partial normalization of a solitary double point, as $p_2=c(p_1),$ 
the compatibility condition amounts to the commutativity of the following diagram:
\begin{equation*}
\xymatrix{
\GG_{p_2}\ar[r]^c & {\ol \GG_{p_1}} \ar[d]^{\bar \lambda}\\
\GG_{p_1}\ar[u]^{\lambda} \ar[r]^c & {\ol \GG_{p_2}} }
\end{equation*}
which yields $\lambda\bar \lambda=1,$ i.e., $\lambda \in S^1.$

In the case of a partial normalization of a cross point, as 
$p_1=c(p_1)$ and $p_2=c(p_2),$ the compatibility condition 
amounts to the commutativity of the following diagram:
\begin{equation*}
\xymatrix{
\GG_{p_1}\ar[d]_{\lambda}\ar[r]^c & {\ol \GG_{p_1}}\ar[d]^{\bar \lambda} \\
\GG_{p_2}\ar[r]^c & {\ol \GG_{p_2}} }
\end{equation*}
which yields $\lambda=\bar \lambda,$ i.e., $\lambda \in \RR^*.$

In the case of a partial normalization of a pair of conjugate nodes, 
the compatibility condition amounts in permuting the factors in 
$Isom(\GG_{p'_1},\GG_{p'_2})\times Isom(\GG_{c(p'_1)},\GG_{c(p'_2)})$ 
which yields 
$\lambda=(\lambda',\ol{\lambda'})$ with $\lambda'\in\CC^*.$
\end{proof}

\begin{cor} 
\label{e-fiber}
Let $H_{\RR,o}$ denote the fiber of $n^*.$ Then $e(H_{\RR,o})=-2$ 
if $H_{\RR,o}$ is $\RR^*,$ and  $0$ otherwise. \qed
\end{cor}

In the cases that we are interested in, there is no difference between 
$\ol{\Pic}^d_{\RR}(Y)$ and $(\ol{\Pic}^d_{\RR}(Y))_+.$

\begin{lem} 
\label{strong-reality}
Let $Y$ be a real reduced irreducible rational nodal curve of arithmetic genus $g.$
If $d=g \mod 2$, then $\ol{\Pic}^d_{\RR}(Y)=(\ol{\Pic}^d_{\RR}(Y))_+$.
\end{lem}
\begin{proof} The sheaves $\FF\in \ol{\Pic}^d_{\RR,o}(Y)$ being simple, 
we can apply the same arguments as in Lemma \ref{Mukai-stronglyreal}.
After that we need only to exclude the case $Y_\RR$ is empty and $c_\FF$ 
is of order 4. But, $Y_\RR=\emptyset$ implies
$d$ and $g$ to be both even. As is well known (cf., \cite[Proposition 2.2 ]{g-h})
in the case of line bundles, if $g$ is even then $c_\FF$ is of order 2. If $\FF$ is of type $II$, 
then we represent it as $\FF= n_*\FF'$ where $\FF'$ is a line bundle on some 
real partial normalization $Y'$ of $Y$, and the same argument applies, 
since the arithmetic genus of $Y'$ has the same parity as that of $Y$.
\end{proof}

\begin{prop}
\label{Euler-picard}
Let $Y$ be a real nodal reduced irreducible rational curve of arithmetic genus $g.$ Then 
$$
e\left(\ol{\Pic}^g_{\RR}(Y)\right)=(-1)^r,
$$
where $r$ is the number of cross-points of $Y_\RR.$
\end{prop}

\begin{proof} From Lemma \ref{type1:embed} we have an obvious decomposition,
$$
\ol{\Pic}^g_{\RR}(Y)=\ol{\Pic}^g_{\RR,b}(Y) \coprod \ol{\Pic}^g_{\RR,o}(Y).
$$
Let $p:Y'\ra Y$ denote the partial normalization of either a cross-point, or of a solitary point. 
Using the additivity of the Euler characteristic and Lemmas \ref{rbs} and \ref{ros}, we have
\begin{align*}
e\left(\ol{\Pic}^g_{\RR}(Y)\right)=&e\left(\ol{\Pic}^g_{\RR,b}(Y)\right)+e\left(\ol{\Pic}^g_{\RR,o}(Y)\right)\\=
&e\left(\ol{\Pic}^{g-1}_{\RR}(Y')\right)+e(H_{\RR,o})e\left(\ol{\Pic}^g_{\RR}(Y')\right)
\end{align*}

If $n$ is the partial normalization of a cross-point,  from Corollary \ref{e-fiber} we have 
$$
e\left(\ol{\Pic}^g_{\RR}(Y)\right)=e\left(\ol{\Pic}^{g-1}_{\RR}(Y')\right)-2e\left(\ol{\Pic}^g_{\RR}(Y')\right).
$$
However, if $Y_\RR$ has a cross point, then its smooth locus is non-empty, and so 
the smooth locus of $Y'_\RR$ is non-empty, as well. But, in this case, tensoring 
by the ideal sheaf of a smooth real point induces an isomorphism between 
$\ol{\Pic}^{g}_{\RR}(Y')$ and $\ol{\Pic}^{g-1}_{\RR}(Y').$ 
Therefore $e\left(\ol{\Pic}^g_{\RR}(Y)\right)=-e\left(\ol{\Pic}^{g-1}_{\RR}(Y')\right).$

In the case of a solitary point, from Corollary \ref{e-fiber} we get 
$e\left(\ol{\Pic}^g_{\RR}(Y)\right)=e\left(\ol{\Pic}^{g-1}_{\RR}(Y')\right).$

If $n:Y'\ra Y$ is the partial normalization of a pair of conjugate imaginary nodes, 
similar arguments show that 
$$
e\left(\ol{\Pic}^g_{\RR}(Y)\right)=e\left(\ol{\Pic}^{g-2}_{\RR}(Y')\right),
$$

Iterating the partial normalization procedure until reaching the full normalization 
$p:\widetilde Y\ra Y,$ we get 
$$
e\left(\ol{\Pic}^g_{\RR}(Y)\right)=
(-1)^{r}e\left(\ol{\Pic}^{g-r-s-2t}_{\RR}(\widetilde Y)\right)=
(-1)^{r}e\left({\Pic}^{0}_{\RR}(\widetilde Y)\right).
$$
Since $g(\widetilde Y)=0,$ the Jacobian variety $\Pic^0_\RR(\widetilde Y)$ 
consists of the trivial line bundle only,
 regardless on whether $\widetilde Y_\RR=\emptyset$ or not, and the proof of
 Proposition \ref{Euler-picard} follows.
\end{proof}

\section{Proof of Theorem \ref{main}}

If $g=1,$ then $X$ is equipped with a generic real elliptic pencil $X\to \PP_\CC^1$ and the result 
follows from Viro-Fubini theorem \cite{viro} applied to $X_\RR\to \PP^1_\RR$.

Let $g\ge 2$. According to Chen's theorem \cite{chen}, on a generic complex $K3$ surface of 
Picard rank $1$ all the rational curves in a primitive class are nodal. In our Theorem \ref{main}, 
we consider real $K3$ surfaces of Picard rank $1$ that are generic in Chen's sense\begin{footnote}
{As is known, on any $n$-dimensional complex analytic variety with an anti-holomorphic involution 
each proper complex analytic subset cuts a proper real analytic subset on every $n$-dimensional 
stratum, or connected component, of the real locus. Applying this argument to the moduli, or period, 
spaces of polarized $K3$ surfaces shows that real $K3$ surfaces with N\'eron-Severi group 
isomorphic to $\ZZ$ and generic in Chen's sense do exist in any real deformation class.}\end{footnote}. 
The real $g$-dimensional linear system of curves of genus $g,$  that appears in Theorem \ref{main} 
and which we denote by $\cal C$, is the linear system of curves in the divisor class generating 
the Picard group. Therefore, by Chen's theorem, the real rational curves in this linear system 
are all nodal. Furthermore, they are reduced and irreducible.

Since the number of cross points in the real locus of a real nodal
rational curve has the same parity as the genus of the curve minus 
the number of solitary points, we get Theorem \ref{main} by performing the 
fiber-wise integration over the Euler characteristic \cite{viro}
applied to the natural projection $\ol{\Pic}_\RR^g{\cal C}\to |\cal C|_\RR,$ 
which due to Propositions \ref{nullity} and \ref{Euler-picard} gives us 
$w_g=(-1)^g e(\ol{\Pic}_\RR^g{\cal C}),$ where $w_g$ is the number of 
real rational curves in the linear system $\cal C$ counted with Welschinger sign.
It suffices now to check the following statement.

\begin{prop}
\label{PicardToHilbert}
$\ol{\Pic}_\RR^g{\cal C}$ and $X^{[g]}_\RR$ are diffeomorphic.
\end{prop}

\begin{proof} Start from considering the natural birational map 
$\phi: \ol{\Pic}^g{\cal C} \to X^{[g]}$, see \cite[Proposition 1.3]{beauville}.
Recall that on the  open subset of $\ol\Pic ^g{\cal C},$ consisting of pairs $(C,L),$
where $C\in \cal C$ is  smooth, $L$ is invertible, and 
$\dim H^0(C,L)=1$, this map sends such a pair to the unique 
effective degree $g$ divisor $D$ on $C$ with $L=\cal O_C(D)$, where
then $D$ is considered as a length-$g$ subscheme of $X$. 
As is immediate from such a description, this map is equivariant with respect
to complex conjugation, that is $\phi\circ c_{\ol{\Pic}^g}=c_{X^{[g]}}\circ\phi$ 
at each point where $\phi$ is well defined, and 
$c_{\ol{\Pic}^g}\circ\phi^{-1}=\phi^{-1}\circ c_{X^{[g]}}$ 
at each point where $\phi^{-1}$ is well defined.

As is known (see, for example Lemma 2.6 \cite{huybrechts}), 
the birational map $\phi$ induces an isomorphism 
$H^2(\ol{\Pic}^g{\cal C} ;\ZZ) \to H^2( X^{[g]};\ZZ) $ 
which is compatible with the Hodge structures and 
Beauville-Bogomolov quadratic forms. In addition, according to above, 
it commutes with the action of the complex conjugation.
Such an isomorphism $H^2(\ol{\Pic}^g{\cal C} ;\ZZ) \to H^2( X^{[g]};\ZZ) $  
provides an equivariant isomorphism between the local universal deformation 
spaces of $Z=\ol{\Pic}^g{\cal C}$ and $Z'=X^{[g]}.$ The latter spaces are smooth
and equipped with a standard action of the complex conjugation 
(up to an appropriate choice of local coordinates),  the fixed points of this action 
form the local universal spaces of equivariant deformations for $(Z,c)$ and $(Z',c)$,
and the above isomorphism being equivariant with respect to complex conjugation 
induces an isomorphism between these spaces of equivariant deformations.

In thus identified local universal spaces of equivariant deformations, 
pick a smooth real curve germ $S=\{s_t\}_{t\in(\RR,0)}$ that corresponds to a 
such generic equivariant deformation of $(Z,c)$ and $(Z',c')$ that the 
Picard number of $Z_t$ and $Z'_t$ becomes $0$ for all, except countably many, $t\ne 0$.
Then, $Z_t$ and $Z'_t$ with non exceptional values of $t\ne 0$ are isomorphic over 
$\CC$ (as it follows, for example, from \cite[Theorems 1.3 and 5.1]{huy-bourbaki}). 
This isomorphism provides a diffeomorphism between the real loci 
of $Z_t$ and $Z'_t$ as soon as it is commuting with the real structures. 
Clearly, such a commutativity holds, if  $Z_t$, or $Z'_t,$ has no nontrivial automorphisms.

So, to finish the proof it is sufficient to show that any hyperk\"ahler manifold $Z$ 
with odd second Betti number (in our case, it is $23$) and Picard number $0$ 
has no any nontrivial automorphism. For that, we adapt the arithmetical arguments 
traditional for proving similar statements in the case of $K3$ surfaces 
(see, for example, \cite{Nikulin} and \cite[Appendix D]{DIK}).

Let $\omega$ be a holomorphic $2-$form generating $H^0(Z,\Omega^2).$
First, if  $g^*\omega=\omega$, then since Bogomolov-Beauville-Fujiki product with 
$\omega$ embeds $H^2(Z,\ZZ)$ into $\CC$ under assumption that the Picard number is 0,
the action of $g$ on $H^2(Z,\ZZ)$ is trivial. Thus, $g$ is the identity {\it cf.}, \cite[Remark 4.2]{oguiso-2}.
Second, once more using the Bogomolov-Beauville-Fujiki pairing embedding in a similar fashion, 
we observe that if $(g^k)^*\omega\ne \omega$, then $(g^k)^*$ has no nonzero fixed element in 
$H^2(Z,\ZZ)$. Hence, $g^*,$ and therefore $g,$  can not be of finite order;  indeed, otherwise, 
according to the latter observation, either $g^*=-1$, which contradicts to preservation of the 
K\"ahler cone, or the free abelian group $H^2(Z,\ZZ)$ will split in some number of irreducible 
presentations of rank $\phi(n)$ of a cyclic group $\ZZ/n$ with $n>2$ 
(here, $\phi$ stands for Euler's totient function; as is known, $\phi(n)$ is the maximal possible rank), 
which is impossible since $\rank H^2(Z,\ZZ)=23$ and $\phi(n)$ is even for any integer $n>2.$ 
As a conclusion, the automorphism group is a free abelian group. To show that its rank is $0,$ 
we can either adapt D.2.9 from \cite{DIK}  (just taking into account the change of parity of $b_2$ 
with respect to the case of $K3$ surfaces) or apply \cite[Corollary 1.6]{oguiso-1}.
\end{proof}

\section{Proof of Theorem \ref{realG}}
\label{sec:real.G}

\subsection{Local punctual Hilbert scheme} The complex conjugation acts naturally on 
the local punctual Hilbert schemes $\hilb^n_{\CC}(0)$ of $\CC^2$ with support at $0$ 
(we abbreviated the usual notation $\hilb^n(\CC^2,0)$ for shortness). This action preserves
the cellular decomposition of $\hilb_{\CC}^n(0)$ constructed by Ellingsrud-Str\o mme \cite{e-s}.
More precisely, this cellular decomposition is based on Bia\l yni\-cki-Birula one and is associated 
with a generic one-parameter subgroup of the torus action on $\hilb_{\CC}^n(0)$ induced by 
the standard coordinate-wise torus action on $\CC^2$; this toric action on $\hilb_{\CC}^n(0)$ 
has a finite number of fixed points, all being real, and it commutes with the complex conjugation.
Therefore: each cell of this decomposition is isomorphic to a complex affine space $\CC^k$, $k\in\NN;$ 
it is invariant under complex conjugation; and its real part, which is isomorphic respectively to a 
real affine space $\RR^k,$ provides a real cell of a cellular decomposition of the real part 
$\hilb^n_{\RR} (0)$ of $\hilb_{\CC}^n(0).$ As a consequence, we can easily use Lemma 2.9 
of G\"ottsche \cite{gottsche} or Theorem 3.3.3 of Cheah \cite{cheah-cells} to find the generating series 
for the Euler characteristic of $\hilb^n_\RR (0).$

\begin{prop}
\label{slava's-conj-proof}
We have
$$
\sum_{n\geq 0} e(\hilb^n_\RR (0))t^n=\prod_{k\geq 1}\left( \frac {1}{1+(-x)^k}\right).
$$
\end{prop}

\begin{proof}
Let $p(a,b)$ denote the number of partitions of $a$ whose largest part is $b.$ The arguments of 
\cite{e-s} show that the number of $2i-$cells of $\hilb_{\CC}^n(0)$ is $p(n,n-i),$ and so the number 
of $i-$cells of  $\hilb^n_\RR(0)$ is $p(n,n-i).$ As in the complex case, applying Euler's generating series
(cf., Lemma 2.9 in \cite{gottsche} and Theorem 3.3.3 in \cite{cheah-cells}) we get
\begin{align*}
B(x,y)=&~\sum_{n\geq 0}\sum_{i\geq 0} b_i \left(\hilb^n_\RR(0) \right)y^ix^n=
\sum_{n\geq 0}\sum_{i\geq 0} p(n,n-i) y^ix^n\\ \notag
=&~\prod_{k\geq 1}\left( \frac {1}{1-x^ky^{k-1}}\right).
\end{align*}
As a consequence, the generating series for the Euler characteristic of the real local punctual Hilbert scheme 
$\hilb^n_\RR (0)$ is 
$$
B(x,-1)=\sum_{n\geq 1} e(\hilb^n_\RR(0))x^n=\prod_{k\geq 1}\left( \frac {1}{1+(-x)^k}\right).
$$
\end{proof}

\subsection{Proof of Theorem \ref{realG}}

An anti-holomorphic involution $c$ on a complex surface $X$ induces, in a natural 
component-wise way, an involution on the symmetric product $S^n(X)$. 
The fixed point locus of the latter, which we denote (as it would be natural in the case of $X$ 
being the complexification of an algebraic surface defined over the reals) by $S^n_\RR(X),$ is
distinct from $S^n(X_\RR),$ where $X_\RR$ is the fixed locus of $c.$ Instead, for any $n\geq 1,$ 
the following evident decomposition holds: 
\begin{equation}
\label{sym-conj}
S_\RR^n(X)
=\coprod_{a+2b=n} S^a(X_\RR)\times S^b(X'_\CC), \quad X'_\CC= 
(X \setminus X_\RR)/{conj}.
\end{equation}

This decomposition lifts to the following decomposition of the real loci of 
Hilbert schemes that respects the Hilbert-Chow map:
$$
X^{[n]}_\RR=
\sum_{\substack{\alpha, \beta : \\ \vert\alpha\vert +2\vert\beta\vert=n}} \HH^\alpha(X_\RR)\times \HH^\beta(X'_\CC),
$$
$$
\HH^\alpha(X_\RR)= S^{\alpha}_0(X_\RR)\ltimes \prod_i (\hilb_\RR^{i}(0))^{\alpha_i},
\HH^\beta(X'_\CC)= S^{\beta}_0(X'_\CC)\ltimes \prod_j(\hilb_\CC^{j}(0))^{\beta_j}.
$$
where $\alpha=(\alpha_1,\dots )$ and $\beta=(\beta_1, \dots)$
denote finite sequences of positive integers, $\vert \cdot\vert$ denotes the $L_1$-norm, 
and $ B \ltimes F$ stands for fiber bundles with base $B$ and fiber $F.$
The symbol  $S_0^\gamma (V)$ with $ \gamma=(\gamma_1,\dots, \gamma_l)$ 
denotes the main stratum of a $\gamma$-symmetric power of a space $V:$
$$
S_0^\gamma(V)=
(V^{\vert\gamma\vert} \setminus \Delta)/(\Sigma_{\gamma_1}\times\dots\times\Sigma_{\gamma_l}).
$$
Here, $\Delta =\{x\in V^{\vert\gamma\vert} : x_i\ne x_j \, \text{for}\, i\ne j\}$ and 
$\Sigma_{\gamma_i},~i=1,\dots, l,$ are  the symmetric groups over $\gamma_i$ elements 
that act in $V^{\vert\gamma\vert}$ by permuting the corresponding coordinates. 

The fiber bundle structure in the above decomposition implies that the generating function
for the numerical sequence $e(X^{[n]}_\RR)$ naturally splits in a product of two generating 
functions, one with terms coming from 
$\HH^{\alpha} (X_\RR)$,
$$
f_\RR=\sum e(\HH^{\alpha}(X_\RR))q^{I(\alpha)}, \quad I(\alpha)=\sum i\alpha_i,
$$
and one with terms coming from $\HH^\beta(X'_\CC),$
$$
f_\CC=\sum e(\HH^{\beta}(X'_\CC))q^{I(\beta)}, \quad I(\beta)=\sum j\beta_j.
$$
The structure of the latter one is exactly the same as in the motivic formula for the punctual Hilbert 
schemes of complex surfaces (see, for example, \cite{cataldo} or \cite{gottsche-motive}), 
so the standard proofs apply and give us the formula
$$
f_\CC= \prod_{s\geq 1} \frac{1}{(1-q^{2s})^{\frac{e_\CC-e_{\RR}}2}}.
$$
As to $f_\RR,$ it gets the following expression 
$$
f_\RR=\sum_{l,\alpha=(\alpha_1,\dots,\alpha_l)}
e(S^\alpha_0(X_\RR))e(1)^{\alpha_1}\dots e(l)^{\alpha_l}q^{I(\alpha)},
$$
where
$
 e(S^\alpha_0(X_\RR))=\frac{e_{\RR}(e_{\RR}-1)\dots
(e_\RR-\vert\alpha\vert+1)}
{ \alpha_1 !\dots \alpha_l!}
$
and 
$e(k)=e(\hilb_\RR^{k}(0)).$
According to Proposition \ref{slava's-conj-proof},
$$
\sum e(k)q^k=\prod_{k\geq 1}\left( \frac {1}{1+(-q)^k}\right).
$$
Hence (compare the proof of Theorem 2.1 in \cite{cataldo}),
$$
f_\RR = \prod_{r\geq 1} \frac{1}{\left(1+(-q)^r\right)^{e_{\RR}}},
$$
and the formula (\ref{real-gottsche}) follows.
\qed

\ps

From Theorems \ref{main} and \ref{realG}, the first part of Corollary \ref{ryz} follows 
immediately,  while the second part follows by substituting $-q$ for $q.$

\begin{rmk}
For comparison with the punctual Hilbert schemes, let us point that in the case of 
the symmetric products the generating function has the following shape:
$$
\displaystyle \sum_{n\geq 0} e(S^n_\RR(X))t^n
= \frac {1}{(1-t)^{e_{\RR}}(1-t^2)^{(e_\CC-e_{\RR})/2}},
$$
as it can be easily seen from (\ref{sym-conj}) and Macdonald's formula \cite{macdonald}.
\end{rmk}

\section{Examples}
\label{ER}

\subsection{A few special cases} 

Here is the first values of $e(X^{[g]}_\RR)$ encoded in formula (\ref{real-gottsche}):

\begin{align*}
e(X^{[1]}_ \RR )=&e_{\RR}\\
e(X^{[2]}_ \RR ))=&\frac12(e_{\RR}^2+e_\CC)-e_{\RR}\\
e(X^{[3]}_\RR)=&\frac 16 e_\RR(e_\RR-1)(e_\RR-2)+  \frac12( e_\CC e_\RR-e_\RR^2+2e_\RR)\\
e(X^{[4]}_\RR)=&\frac1{24}e_{\RR}(e_{\RR}-1)(e_{\RR}-2)(e_{\RR}-3) \\ \notag
+ &\frac1{8}(e_\CC^2 +2e_\CC e_{\RR}^2-4e_\CC e_{\RR}-2e_{\RR}^3+6e_\CC+11e_{\RR}^2-6e_{\RR})
\end{align*}

\subsection{Sharpness} Applying the above formulas to $K3$ surfaces with an elliptic pencil 
(case $g=1$), to curves of degree $6$ in the projective plane (case $g=2$) and surfaces of 
degree $4$ in the projective $3-$space (case $g=3$), we deduce from Theorem \ref{main} that:
\begin{itemize}
\item[ (1)] {\it for any generic real elliptic $K3$ surface, the elliptic pencil contains at least 
$\vert e(X_\RR)\vert$ real rational curves,}
\item[ (2)]
{\it any generic real nonsingular plane curve of degree $6$ has at least $12$ real double tangents
(including the real tangents at a pair of imaginary complex conjugate points) and if its real locus 
consists of $10$ ovals placed outside each other, then there are at least $192$ real double tangents} 
(recall that $10$ is the maximal possible number of ovals when they are placed outside each other), 
\item[ (3)]
{\it any generic real nonsingular surface of degree 4 in the projective $3-$space of Harnack type 
{\rm (that is two-component surfaces with one component  of genus $0$ and one of genus $10$)}  
has at least $1152$ real triple tangent planes, and those whose real locus 
is connected and of genus $10$ {\rm (thus, of Euler characteristic $-18$)} have at least $1536$ 
real triple tangent planes} (including the real planes tangent at a real point and a pair of imaginary 
complex conjugate points; $-18$ is the most negative value for the Euler characteristic of 
real quartic surfaces, and $10$ is the maximal value for the genus). 
\end{itemize}

The sharpness of the first two bounds is easy to establish by elementary means (in particular, 
the existence of real elliptic $K3$ surfaces with connected real locus of Euler characteristic 
$-18$ and $18$ real singular fibers implies the sharpness of the first bound in the case 
$g=1, e(X_\RR)=-18$); for the third bound in what concerns Harnack surfaces it can be 
established by means of tropical geometry, as was kindly communicated to us by Grigory Mikhalkin.

Since the real rational curves without real points give positive inputs into our count, 
the bound given in Theorem \ref{main} is sharp also for all real $K3$ surfaces with the empty real locus.

\subsection*{Acknowledgements} The second author acknowledges the support of the 
Simons Foundation's "Collaboration Grant for Mathematicians" 
and the hospitality of the IH\'ES and IRMA Strasbourg, where this project was launched. 
The first author acknowledges the hospitality of the MPIM Bonn, where further advances
were achieved. Our thanks to Dropbox, which made our work on distance comfortable.

\bibliographystyle{alpha} 

\end{document}